%% file: doc.tex
\documentclass[a4paper,11pt,reqno]{customart}

\usepackage[utf8x]{inputenc}
\usepackage[margin=1in]{geometry}
\usepackage{verbatim}
\usepackage{color}
\usepackage[table]{xcolor}
\usepackage{listings}
\usepackage{amssymb,amsfonts,amsthm,amsmath}
\usepackage{url}
\usepackage{enumerate}
\usepackage{todonotes}
\usepackage[all]{xy}
\usepackage{multirow}
\usepackage{tikz}
\usepackage[underline=false]{pgf-umlsd}
\usepackage{stmaryrd}
\usepackage{footnote}
\makesavenoteenv{tabular}
\usepackage{hyperref}

\makeatletter
\newcommand\footnoteref[1]{\protected@xdef\@thefnmark{\ref{#1}}\@footnotemark}
\makeatother

\hypersetup{
	pdfborder={0 0 0}
}

\definecolor{grey}{rgb}{0.95,0.95,0.95}
\definecolor{green}{rgb}{0.2,0.6,0.4}

\renewcommand{\O}{\textbf{0}}

\newcommand{\Psf}{\mathsf{P}}
\newcommand{\Qsf}{\mathsf{Q}}

\newcommand{\imp}{\rightarrow}

\newcommand{\biimp}{\leftrightarrow}

\newcommand{\Nb}{\mathbb{N}}

\newcommand{\dbf}{\mathbf{d}}
\newcommand{\ebf}{\mathbf{e}}

\newcommand{\Ccal}{\mathcal{C}}

\newcommand{\Mcal}{\mathcal{M}}

\renewcommand{\setminus}{\smallsetminus}

\newcommand{\card}[1]{\left| #1 \right|}

\newcommand{\tuple}[1]{\left\langle #1 \right\rangle}

\newcommand{\s}[1]{\ensuremath{\sf{#1}}}

\newcommand{\dnrs}[1]{#1\mbox{-}\s{DNC}}
\newcommand{\rwkls}[1]{#1\mbox{-}\s{RWKL}}
\newcommand{\gens}[1]{#1\mbox{-}\s{GEN}}
\newcommand{\seqs}[1]{\s{Seq}^{*}(#1)}
\DeclareMathOperator{\rca}{\s{RCA}_0}
\DeclareMathOperator{\aca}{\s{ACA}_0}
\DeclareMathOperator{\wkl}{\s{WKL}_0}

\DeclareMathOperator{\dnr}{\s{DNC}}
\DeclareMathOperator{\bst}{\s{B}\Sigma^0_2}

\DeclareMathOperator{\rwkl}{\s{RWKL}}

\DeclareMathOperator{\rt}{\s{RT}}
\DeclareMathOperator{\srt}{\s{SRT}}

\DeclareMathOperator{\ads}{\s{ADS}}

\DeclareMathOperator{\rcolor}{\s{RCOLOR}}

\DeclareMathOperator{\coh}{\s{COH}}

\DeclareMathOperator{\fip}{\s{FIP}}
\DeclareMathOperator{\ts}{\s{TS}}
\DeclareMathOperator{\sts}{\s{STS}}

\DeclareMathOperator{\ipt}{\s{IPT}}

\DeclareMathOperator{\fs}{\s{FS}}

\DeclareMathOperator{\emo}{\s{EM}}
\DeclareMathOperator{\semo}{\s{SEM}}

\usetikzlibrary{shapes,arrows}
\usetikzlibrary{decorations.markings}
\definecolor{lightblue}{HTML}{e6e6e6}
\definecolor{lightred}{HTML}{eca6a6}
\definecolor{lightgreen}{RGB}{164,244,140}

\newtheoremstyle{custom}
  {10pt}
  {10pt}
  {\normalfont}
  {}
  {\bfseries}
  {}
  { }
  {}

\theoremstyle{custom}

\usepackage{xcolor}	
\usepackage{soul}

\newcommand{\update}[1]{\textcolor{green}{UPDATE: #1}}

\newtheorem{theorem}{Theorem}[section]
\newtheorem{lemma}[theorem]{Lemma}
\newtheorem{definition}[theorem]{Definition}

\newtheorem{question}[theorem]{Question}

\newtheorem{corollary}[theorem]{Corollary}

\begin{document}

\title[Open questions about Ramsey-type statements in reverse mathematics]
	{Open questions about Ramsey-type statements\\ in reverse mathematics}

\author{
  Ludovic Patey
}

\begin{abstract}
Ramsey's theorem states that for any coloring of the $n$-element subsets of $\Nb$ with
finitely many colors, there is an infinite set~$H$ such that all $n$-element subsets of
$H$ have the same color. The strength of consequences of Ramsey's theorem
has been extensively studied in reverse mathematics 
and under various reducibilities, namely, computable reducibility and uniform reducibility.
Our understanding of the combinatorics of Ramsey's theorem and its consequences
has been greatly improved over the past decades.
In this paper, we state some questions which naturally arose during 
this study. The inability to answer those questions reveals
some gaps in our understanding of the combinatorics of Ramsey's theorem.
\end{abstract}

\maketitle

\begin{center}
\textcolor{green}{Last update: December 2024}
\end{center}

\section{Introduction}
\input{parts/part0-introduction}

\section{Cohesiveness and partitions}\label{sect:partitions-reducibility}
\input{parts/part1-coh-srt}

\section{A Ramsey-type weak K\"onig's lemma}
\input{parts/part2-rwkl}

\section{The Erd\H{o}s-Moser theorem}
\input{parts/part3-em}

\section{Ramsey-type hierarchies}
\input{parts/part4-hierarchies}

\section{Proofs}
\input{parts/part9-appendix}

\vspace{0.5cm}

\noindent \textbf{Acknowledgements}. The author is thankful to his PhD advisor Laurent Bienvenu
for interesting comments and discussions.
The author is funded by the John Templeton Foundation (`Structure and Randomness in the Theory of Computation' project). The opinions expressed in this publication are those of the author(s) and do not necessarily reflect the views of the John Templeton Foundation.

\vspace{0.5cm}

\bibliographystyle{plain}
\bibliography{../bibliography}

\end{document}

%% file: parts/part0-introduction.tex
Ramsey's theory is a branch of mathematics studying the conditions
under which some structure appears among a sufficiently large collection of objects.
Perhaps the most well-known example is Ramsey's theorem, which states that for any coloring of the $n$-element subsets of $\Nb$ with
finitely many colors, there is an infinite set~$H$ such that all $n$-element subsets of
$H$ have the same color. Consequences of Ramsey's theorem have been extensively studied in reverse mathematics 
and under various reducibilities, among which, computable reducibility.

Reverse mathematics is a vast mathematical program
whose goal is to classify ordinary theorems in terms of their provability strength.
It uses the framework of subsystems of second-order arithmetic,
which is sufficiently rich to express many theorems in a natural way.
The base system, $\rca$ (standing for Recursive Comprehension Axiom),
contains basic first-order Peano arithmetic together with the~$\Delta^0_1$
comprehension scheme and the~$\Sigma^0_1$ induction scheme.
The early study of reverse mathematics revealed that most ``ordinary'', i.e., non set-theoretic, theorems
are equivalent to five main subsystems, known as the ``Big Five''~\cite{Simpson2009Subsystems}.
However, Ramsey's theory provides a large class of theorems escaping this observation,
making it an interesting research subject in reverse mathematics~\cite{bienvenu2017logical,Cholak2001strength,Hirschfeldt2008strength,Hirschfeldt2007Combinatorial,patey2015combinatorial,Seetapun1995strength}.
The book of Hirschfeldt~\cite{hirschfeldt2015slicing} is an excellent introduction to reverse mathematics and in particular
the reverse mathematics of Ramsey's theorem.

Many theorems are $\Pi^1_2$ statements~$\Psf$ of the form $(\forall X)[\Phi(X) \imp (\exists Y)\Psi(X,Y)]$.
A set~$X$ such that~$\Phi(X)$ holds is called a \emph{$\Psf$-instance} and a set $Y$ such that~$\Psi(X,Y)$
holds is a \emph{solution to $X$}. A theorem~$\Psf$ is computably reducible to~$\Qsf$ (written~$\Psf \leq_c \Qsf$)
if for every $\Psf$-instance~$X$, there is an~$X$-computable $\Qsf$-instance~$Y$ such that
every solution to~$Y$ computes relative to~$X$ a solution to~$X$.
Computable reducibility provides a more fine-grained analysis
of theorems than reverse mathematics, in the sense that it is sensitive to the number of applications
of the theorem~$\Psf$ in a proof that $\Psf$ implies~$\Qsf$ over~$\rca$~\cite{dorais2016uniform,dzhafarov2016strong,hirschfeldt2016notions,Patey2016weakness}.
For example, Ramsey's theorem for~$(k+1)$-colorings of the~$n$-element subsets of~$\Nb$
is not computably equivalent to Ramsey's theorem for~$k$-colorings, whereas those statements
are equivalent over~$\rca$~\cite{Patey2016weakness}.

\subsection{Ramsey's theorem}

The strength of Ramsey-type statements is notoriously hard to tackle
in the setting of reverse mathematics. The separation of Ramsey's theorem for pairs ($\rt^2_2$)
from the arithmetic comprehension axiom ($\aca$) was a long-standing open
problem, until Seetapun solved it~\cite{Seetapun1995strength} using the notion of cone avoidance.

\begin{definition}[Ramsey's theorem]
A subset~$H$ of~$\Nb$ is~\emph{homogeneous} for a coloring~$f : [\Nb]^n \to k$ (or \emph{$f$-homogeneous}) 
if all $n$-tuples over~$H$ are given the same color by~$f$. 
$\rt^n_k$ is the statement ``Every coloring $f : [\Nb]^n \to k$ has an infinite $f$-homogeneous set''.
\end{definition}

Jockusch~\cite{Jockusch1972Ramseys} studied the computational strength of Ramsey's theorem,
and Simpson~\cite[Theorem III.7.6]{Simpson2009Subsystems} built upon his work to prove that whenever~$n \geq 3$ and~$k \geq 2$,
$\rca \vdash \rt^n_k \biimp \aca$. Ramsey's theorem for pairs is probably the most famous 
example of a statement escaping the Big Five. Seetapun~\cite{Seetapun1995strength} proved that~$\rt^2_2$ 
is strictly weaker than~$\aca$ over~$\rca$. Because of the complexity of the related separations, 
$\rt^2_2$ received particular attention from the reverse mathematics community~\cite{Cholak2001strength,
Seetapun1995strength,Jockusch1972Ramseys}. 
Cholak, Jockusch and Slaman~\cite{Cholak2001strength} and Liu~\cite{liu2012rt2_2} proved that 
$\rt^2_2$ is incomparable with weak K\"onig's lemma.
Dorais, Dzhafarov, Hirst, Mileti and Shafer~\cite{dorais2016uniform},
Dzhafarov~\cite{Dzhafarov2014Cohesive}, Hirschfeldt and Jockusch~\cite{hirschfeldt2016notions},
Rakotoniaina~\cite{Rakotoniaina2015Computational}
and the author~\cite{Patey2016weakness} studied the computational strength of Ramsey's theorem
according to the number of colors, when fixing the number of applications of the principle.

In this paper, we state some remaining open questions which naturally arose during
the study of Ramsey's theorem. Many of them are already stated in various papers.
These questions cover a much thinner branch of reverse mathematics than
the paper of Montalban~\cite{Montalban2011Open} and are of course
influenced by the own interest of the author. We put our focus on a few central open questions,
motivate them and try to give some insights about the reason for their complexity. 
We also detail some promising approaches and ask
more technical questions, driven by the resolution of the former ones.
The questions are computability-theoretic oriented, and therefore mainly
involve the relations between statements over $\omega$-models,
that is, models where the first-order part is composed of the standard integers.

%% file: parts/part1-coh-srt.tex
In order to better understand the combinatorics of Ramsey's theorem for pairs,
Cholak et al.~\cite{Cholak2001strength} decomposed it into a cohesive and a stable version.
This approach has been fruitfully reused in the analysis of various consequences
of Ramsey's theorem~\cite{Hirschfeldt2007Combinatorial}.

\begin{definition}[Stable Ramsey's theorem]
A function~$f : [\Nb]^2 \to k$ is \emph{stable} if for every~$x \in \Nb$,
$\lim_s f(x,s)$ exists. $\srt^2_k$ is the restriction of~$\rt^2_k$ to stable colorings.
$\mathsf{D}^2_k$ is the statement ``Every $\Delta^0_2$ $k$-partition has an infinite
subset in one of its parts''.
\end{definition}

By Shoenfield's limit lemma~\cite{Shoenfield1959degrees},
a stable coloring $f : [\Nb]^2 \to 2$ can be seen as the $\Delta^0_2$
approximation of the $\Delta^0_2$ set~$A = \lim_s f(\cdot, s)$.
Cholak et al.~\cite{Cholak2001strength} proved that~$\mathsf{D}^2_k$ and~$\srt^2_k$ are computably equivalent
and that the proof is formalizable over~$\rca+\bst$, where~$\bst$ is the $\Sigma^0_2$-bounding statement.
Later, Chong et al.~\cite{Chong2010role} proved that~$\mathsf{D}^2_2$ implies~$\bst$ over~$\rca$,
showing therefore that~$\rca \vdash \mathsf{D}^2_k \biimp \srt^2_k$ for every~$k \geq 2$.
The statement~$\mathsf{D}^2_k$ can be seen as a variant of~$\rt^1_k$, where the instances are $\Delta^0_2$.
It happens that many statements of the form ``Every computable instance of~$\srt^2_2$ has a solution satisfying
some properties'' are proven by showing the following stronger statement ``Every instance of~$\rt^1_2$ (without any effectiveness
restriction) has a solution satisfying some properties''. This is for example the case for closed sets avoidances~\cite{Dzhafarov2009Ramseys,Liu2015Cone,liu2012rt2_2}, 
and various preservation notions~\cite{Patey2015Iterative,Patey2016weakness,wang2016definability}.
This observation shows that the weakness of~$\rt^1_2$ has a combinatorial nature,
whereas $\srt^2_2$ is only effectively weak.

\begin{definition}[Cohesiveness]
An infinite set $C$ is $\vec{R}$-cohesive for a sequence of sets $R_0, R_1, \dots$
if for each $i \in \Nb$, $C \subseteq^{*} R_i$ or $C \subseteq^{*} \overline{R_i}$.
A set $C$ is \emph{p-cohesive} if it is $\vec{R}$-cohesive where
$\vec{R}$ is an enumeration of all primitive recursive sets.
$\coh$ is the statement ``Every uniform sequence of sets $\vec{R}$
has an $\vec{R}$-cohesive set.''
\end{definition}

Jockusch \& Stephan~\cite{Jockusch1993cohesive} studied the degrees of unsolvability of cohesiveness
and proved that~$\coh$ admits a universal instance, i.e., an instance whose solutions compute solutions
to every other instance. This instance consists of the primitive recursive sets.
They characterized the degrees of p-cohesive sets as those whose
jump is PA relative to~$\emptyset'$. Later, the author~\cite{Patey2016weakness} refined this correspondence
by showing that for every computable sequence of sets~$R_0, R_1, \dots$,
there is a $\Pi^{0, \emptyset'}_1$ class $\Ccal$ of reals such that the degrees bounding an~$\vec{R}$-cohesive set
are exactly those whose jump bounds a member of~$\Ccal$. Moreover, for every $\Pi^{0,\emptyset'}_1$ class~$\Ccal$,
we can construct a computable sequence of sets~$\vec{R}$ satisfying the previous property.
In particular, if some instance of~$\coh$ has no computable solution, then it has no low solution.
This shows that $\coh$ is a statement about the Turing jump.

Cholak, Jockusch and Slaman~\cite{Cholak2001strength} 
claimed that~$\rt^2_2$ is equivalent to~$\srt^2_2+\coh$ over~$\rca$
with an erroneous proof. Mileti~\cite{Mileti2004Partition} and Jockusch \& Lempp [unpublished]
independently fixed the proof. 
Cholak et al.~\cite{Cholak2001strength} proved that~$\coh$ does not imply~$\srt^2_2$
over~$\rca$. The question of the other direction has been a long-standing open problem.
Recently, Chong et al.~\cite{Chong2014metamathematics} proved that~$\srt^2_2$ is strictly weaker than~$\rt^2_2$
over~$\rca+\bst$. However they used non-standard models to separate the statements and the question
whether~$\srt^2_2$ and~$\rt^2_2$ coincide over~$\omega$-models remains open.
See the survey of Chong, Li and Yang~\cite{Chong2014Nonstandard} for the approach of non-standard analysis 
applied to reverse mathematics.

\begin{question}\label{quest:srt22-rt22-omega}
Does~$\srt^2_2$ imply~$\rt^2_2$ (or equivalently $\coh$) over $\omega$-models?
\end{question}

\update{This question was answered negatively by Monin and Patey in \cite{monin2021srt2_2}}
\smallskip

Jockusch~\cite{Jockusch1972Ramseys} constructed a computable instance of~$\rt^2_2$
with no $\Delta^0_2$ solution. Cholak et al.~\cite{Cholak2001strength} suggested building
an $\omega$-model of~$\srt^2_2$ composed only of low sets. However, Downey et al.~\cite{downey20010_2}
constructed a $\Delta^0_2$ set with no low subset of either it or its complement.
As often, the failure of an approach should not be seen as a dead-end, but as a starting point.
The construction of Downey et al.\ revealed that~$\srt^2_2$ carries some additional computational
power, whose nature is currently unknown. Indeed, all the natural computability-theoretic consequences
of~$\srt^2_2$ known hitherto admit low solutions. Answering the following question 
would be a significant step towards understanding the strength of~$\srt^2_2$.

\begin{question}
Is there a natural computable $\srt^2_2$-instance with no low solution?
\end{question}

Here, by ``natural'', we mean an instance which carries more informational content than having no low solution. 
Interestingly, Chong et al.~\cite{Chong2014metamathematics} constructed a non-standard
model of~$\rca+\bst\allowbreak+\srt^2_2$ containing only low sets. This shows that the argument
of Downey et al.~\cite{downey20010_2} requires more than~$\Sigma^0_2$-bounding,
and suggests that such an instance has to use an elaborate construction.

Hirschfeldt et al.~\cite{Hirschfeldt2008strength} proposed a very promising
approach using an extension of Arslanov's completeness criterion~\cite{Arslanov1981Some,Jockusch1989Recursively}.
A set~$X$ is~\emph{1-CEA over~$Y$} if c.e.\ in and above~$Y$.
A set~$X$ is~\emph{$(n+1)$-CEA over~$Y$} if it is 1-CEA over a set~$Z$ which is $n$-CEA over~$Y$.
In particular the sets which are 1-CEA over~$\emptyset$ are the c.e. sets.
By~\cite{Jockusch1989Recursively}, every set $n$-CEA over~$\emptyset$ computing a set of PA degree is complete.
Hirschfeldt et al.\ asked the following question. 

\begin{question}\label{quest:delta2-low2-delta2}
Does every $\Delta^0_2$ set admit an infinite subset of either it or its complement
which is both low${}_2$ and~$\Delta^0_2$? 
\end{question}

A positive answer to Question~\ref{quest:delta2-low2-delta2} would enable one to build an $\omega$-model $\Mcal$ of~$\srt^2_2$
such that for every set~$X \in \Mcal$, $X$ is low${}_2$ and $X'$ is $n$-CEA over~$\emptyset'$. By Jockusch
\& Stephan~\cite{Jockusch1993cohesive}, if some~$X \in \Mcal$ computes some p-cohesive set, then~$X'$ is of PA degree relative to~$\emptyset'$. By a relativization of Jockusch et al.~\cite{Jockusch1989Recursively}, $X'$ would compute~$\emptyset''$, so would be high,
which is impossible since $X$ is low${}_2$. Therefore, $\Mcal$ would be an $\omega$-model of~$\srt^2_2$ which is not a model of~$\coh$,
answering Question~\ref{quest:srt22-rt22-omega}. Note that the argument also works if we replace~low${}_2$ by low${}_n$, where~$n$ may even
depend on the instance. Hirschfeldt et al.~\cite{Hirschfeldt2008strength} 
proved that every $\Delta^0_2$ set has an infinite incomplete $\Delta^0_2$ subset of either it or its complement. This is the best upper bound currently known. They asked the following question which is the strong negation of Question~\ref{quest:delta2-low2-delta2}.

\begin{question}\label{quest:delta2-jump-pa}
Is there a $\Delta^0_2$ set such that every infinite subset of either it or its complement 
has a jump of PA degree relative to~$\emptyset'$? 
\end{question}

By Arslanov's completeness criterion, a positive answer to Question~\ref{quest:delta2-jump-pa}
would also provide one to the following question.

\begin{question}\label{quest:delta2-high}
Is there a $\Delta^0_2$ set such that every $\Delta^0_2$ infinite subset of either it or its complement 
is high? 
\end{question}

Cohesiveness can be seen as a sequential version of~$\rt^1_2$ with finite errors ($\seqs{\rt^1_2}$).
More formally, given some theorem~$\Psf$, $\seqs{\Psf}$ is the statement ``For every uniform
sequence of $\Psf$-instances~$X_0, X_1, \dots$, there is a set~$Y$ which is, up to finite changes,
a solution to each of the~$X$'s.'' One intuition about the guess that $\srt^2_2$ does not imply~$\coh$
could be that by the equivalence between~$\srt^2_2$ and~$\mathsf{D}^2_2$,
$\srt^2_2$ is nothing but a non-effective instance of~$\rt^1_2$, and that one cannot encode in a single instance
of~$\rt^1_2$ countably many $\rt^1_2$-instances. Note that this argument is of combinatorial
nature, as it does not make any effectiveness assumption on the instance of~$\rt^1_2$.
We express reservations concerning the validity of this argument, as witnessed by what follows.

\begin{definition}[Thin set theorem]
Given a coloring~$f : [\Nb]^n \to k$ (resp.\ $f : [\Nb]^n \to \Nb$), an infinite set~$H$
is \emph{thin} for~$f$ if~$|f([H]^n)| \leq k-1$ (resp.\ $f([H]^n) \neq \Nb$), that is,
$f$ avoids at least one color over~$H$.
For every~$n \geq 1$ and~$k \geq 2$, $\ts^n_k$ is the statement
``Every coloring $f : [\Nb]^n \to k$ has a thin set''
and $\ts^n$ is the statement ``Every coloring $f : [\Nb]^n \to \Nb$ has a thin set''.
\end{definition}

The thin set theorem is a natural weakening of Ramsey's theorem. 
Its reverse mathematical analysis started with Friedman~\cite{FriedmanFom53free,Friedman2013Boolean}.
It has been studied by Cholak et al.~\cite{Cholak2001Free},
Wang~\cite{Wang2014Some} and the author~\cite{patey2015combinatorial,Patey2015Degrees,Patey2016weakness} among others.
According to the definition, $\seqs{\ts^1_k}$ is the statement ``For every uniform sequence of functions
$f_0, f_1, \dots : \Nb \to k$, there is an infinite set~$H$ which is, up to finite changes, thin
for all the~$f$'s.''

\begin{definition}[DNC functions]
A function~$f : \Nb \to \Nb$ is \emph{diagonally non-computable relative to~$X$}
if for every~$e$, $f(e) \neq \Phi^X_e(e)$.
For every~$n \geq 1$, $\dnrs{n}$ is the statement~``For every set~$X$,
there is a function DNC relative to~$X^{(n-1)}$. We write $\dnr$ for~$\dnrs{1}$.
\end{definition}

DNC functions are central notions in algorithmic randomness, as their degrees
coincide with the degrees of infinite subsets of Martin-L\"of randoms
(see Kjos-Hanssen~\cite{kjoshanssen2009infinite} and Greenberg \& Miller~\cite{Greenberg2009Lowness}).
Moreover, Jockusch \& Soare~\cite{Jockusch1972Degrees} and Solovay [unpublished] proved that 
the degrees of $\{0,1\}$-valued DNC functions are exactly the PA degrees.
DNC functions naturally appear in reverse mathematics, the most suprising example
being the equivalence between the rainbow Ramsey theorem for pairs 
and $\dnrs{2}$ proven by Miller~\cite{MillerAssorted}. The rainbow Ramsey theorem
asserts that every coloring of~$[\Nb]^n$ in which each color appears at most $k$ times
admits an infinite set in which each color appears at most once.
It has been studied by Csima \& Mileti~\cite{Csima2009strength}, 
Wang~\cite{WangSome,Wang2014Cohesive} and the author~\cite{Patey2015Somewhere},
among others.

Theorem~3.5 in Jockusch \& Stephan~\cite{Jockusch1993cohesive} can easily be adapted to prove that
the degrees of solutions to primitive recursive instances of~$\seqs{\ts^1}$ (resp.\ $\seqs{\ts^1_k}$)
are exactly those whose jump is of DNC (resp.\ $k$-valued DNC) degree relative to~$\emptyset'$.
By a relativization of Friedberg~\cite{jockusch1989degrees}, the degrees whose jump is PA and those whose
jump bounds a $k$-valued DNC function coincide. Therefore~$\coh$ and~$\seqs{\ts^1_k}$ are computably equivalent.
However, $\seqs{\ts^1}$ is a strictly weaker statement, as for any computable instance of~$\seqs{\ts^1}$,
the measure of oracles computing a solution to it is positive.

Recall our intuition that a single instance of~$\rt^1_2$ cannot encode the information
of countably many instances of~$\rt^1_2$. This intuition is false when considering~$\ts^1$.
Indeed, there is a (non-$\Delta^0_2$) instance of~$\ts^1$ (and \emph{a fortiori} one of~$\rt^1_2$) whose solutions all bound a
function DNC relative to~$\emptyset'$, and therefore computes a solution to any computable instance of~$\seqs{\ts^1}$.
Therefore, before asking whether $\coh$ is a consequence of $\emptyset'$-effective instances of~$\rt^1_2$,
it seems natural to ask whether $\coh$ is a consequence of any coloring over singletons in a combinatorial sense,
that is, with no effectiveness restriction at all.

\begin{question}\label{quest:rt12-combi-jump-pa}
Is there any $\rt^1_2$-instance whose solutions have a jump of PA degree relative to~$\emptyset'$?
\end{question}

Note that a negative answer to Question~\ref{quest:rt12-combi-jump-pa} would have practical reverse mathematical
consequences. There is an ongoing search for natural statements strictly between $\rt^2_2$ and~$\srt^2_2$ over~$\rca$~\cite{Dzhafarov2009polarized}.
Dzhafarov \& Hirst~\cite{Dzhafarov2009polarized} introduced the increasing polarized Ramsey's theorem for pairs ($\ipt^2_2$), 
and proved it to be between~$\rt^2_2$ and~$\srt^2_2$ over~$\rca$. The author~\cite{Patey2015Somewhere} proved that
$\ipt^2_2$ implies the existence of a function DNC relative to~$\emptyset'$, therefore showing that~$\srt^2_2$
does not imply~$\ipt^2_2$ over~$\rca$. The statement $\ipt^2_2$ is equivalent
to $\rwkls{2}$, a relativized variant of the Ramsey-type weak K\"onig's lemma, over~$\rca$,
and therefore is a combinatorial consequence of~$\rt^1_2$. 
An iterable negative answer to Question~\ref{quest:rt12-combi-jump-pa} would prove that~$\ipt^2_2$
does not imply~$\rt^2_2$, hence is strictly between $\rt^2_2$ and~$\srt^2_2$ over~$\rca$.

We have seen that $\seqs{\ts^1}$ is a combinatorial consequence of~$\ts^1$.
This information helps us for tackling the following question.
Indeed, if proven false, it must be answered by effective means and not  by combinatorial ones.

\begin{question}\label{quest:srt-seq-thin-set}
Does~$\srt^2_2$ imply~$\seqs{\ts^1}$ over~$\omega$-models?
\end{question}

\begin{question}\label{quest:rt-jump-DNC}
Is there a $\Delta^0_2$ set such that every infinite subset of either it or its complement 
has a jump of DNC degree relative to~$\emptyset'$? 
\end{question}

Although those questions are interesting in their own right, 
a positive answer to Question~\ref{quest:delta2-low2-delta2} would provide a negative answer to Question~\ref{quest:srt-seq-thin-set},
and a positive answer to Question~\ref{quest:rt-jump-DNC} would provide a positive answer to Question~\ref{quest:delta2-high}.
Indeed, the extended version of Arslanov's completeness criterion states in fact that
every set $n$-CEA over~$\emptyset$ computing a set of DNC degree is complete. In particular,
if Question~\ref{quest:delta2-low2-delta2} has a positive answer, then there is an $\omega$-model of $\srt^2_2$
which is not a model of the rainbow Ramsey theorem for pairs.

Let us finish this section by a discussion about why those problems are so hard to tackle.
There are different levels of answers, starting from the technical one which is more objective, but probably also less informative,
to the meta discussion which tries to give more insights, but can be more controversial.

From a purely technical point of view, all the forcing notions used so far to produce solutions
to Ramsey-type statements are variants of Mathias forcing. In particular, they restrict the future elements
to a ``reservoir''. Any sufficiently generic filters for those notions
of forcing yield cohesive sets. Therefore, one should not expect to obtain a diagonalization against 
instances of~$\coh$ by exhibiting a particular dense set of conditions. Indeed, one would derive a contradiction
by taking a set sufficiently generic to meet both those diagonalizing sets, and the dense sets producing a cohesive solution.
More generally, as long as we use a forcing notion where we restrict the future elements to a reservoir,
any diagonalization against~$\coh$ has to strongly rely on some effectiveness of the overall construction.
The first and second jump control of Cholak et al.~\cite{Cholak2001strength} form a case in point
of how to restrict the amount of genericity to obtain some stronger properties,
which are provably wrong when taking any sufficiently generic filter.

In a higher level, we have mentioned that $\coh$ is a statement
about the jump of Turing degrees. In other words, by Shoenfield's limit lemma~\cite{Shoenfield1959degrees}, 
$\coh$ is a statement about some limit behavior, and is therefore non-sensitive
to any local modification. However, the computability-theoretic properties
used so far to separate statements below~$\aca$ are mainly acting ``below the first jump'',
in the sense that the diagonalization occurs after a finite amount of time.
With $\coh$, there will be some need for a ``continuous diagonalization'', that is, 
a diagonalization which has to be maintained all along the construction.

%% file: parts/part2-rwkl.tex
K\"onig's lemma asserts that every infinite, finitely branching
tree has an infinite path.
Weak K\"onig's lemma ($\wkl$) is the restriction of K\"onig's lemma to infinite binary trees.
$\wkl$ plays a central role in reverse mathematics.
It is one of the Big Five and informally captures compactness arguments~\cite{Simpson2009Subsystems}.
Weak K\"onig's lemma is involved in many constructions of solutions to Ramsey-type statements, e.g.,
cone avoidance~\cite{Seetapun1995strength,Wang2014Some} or control of the jump~\cite{Cholak2001strength,patey2017controlling}.
The question of whether $\rt^2_2$ implies~$\wkl$ over~$\rca$
was open for decades, until Liu~\cite{liu2012rt2_2} solved it by proving that PA degrees
are not a combinatorial consequence of~$\rt^1_2$.

Recently, Flood~\cite{Flood2012Reverse} clarified the relation between Ramsey-type theorems
and $\wkl$, by introducing a Ramsey-type variant of weak K\"onig's lemma ($\rwkl$).
Informally, seeing a set as a 2-coloring of the integers, 
for every $\Pi^0_1$ class of 2-colorings, $\rwkl$ states
the existence of an infinite set homogeneous for one of them.
The exact statement of~$\rwkl$ has to be done with some care, as we do not want 
to state the existence of a member of the $\Pi^0_1$ class.

\begin{definition}[Ramsey-type weak K\"onig's lemma]
A set $H \subseteq \Nb$ is \emph{homogeneous} for a $\sigma \in 2^{<\Nb}$ if $(\exists c < 2)(\forall i \in H)(i < |\sigma| \imp \sigma(i)=c)$, and a set $H \subseteq \Nb$ is \emph{homogeneous} for an infinite tree $T \subseteq 2^{<\Nb}$ if the tree $\{\sigma \in T : \text{$H$ is homogeneous for $\sigma$}\}$ is infinite.  $\rwkl$ is the statement ``for every infinite subtree of $2^{<\Nb}$, there is an infinite homogeneous set.''
\end{definition}

The Ramsey-type weak K\"onig's lemma was introduced by Flood in~\cite{Flood2012Reverse}
under the name $\mathsf{RKL}$, and later renamed $\mathsf{RWKL}$ by Bienvenu, Patey and Shafer.
Flood~\cite{Flood2012Reverse} proved that~$\rwkl$ is a strict consequence of both~$\srt^2_2$ and~$\wkl$,
and that~$\rwkl$ implies~$\dnr$ over~$\rca$. Bienvenu et al.~\cite{bienvenu2017logical} 
and Flood \& Towsner~\cite{flood2016separating} independently separated~$\dnr$ from~$\rwkl$.
They furthermore proved that~$\dnr$ coincides over~$\rca$ to the restriction of~$\rwkl$
to trees of positive measure. 
Very little is currently known about~$\rwkl$.
Despite its complicated formulation, $\rwkl$ is a natural statement
which is worth being studied due to the special status of Ramsey-type theorems in reverse mathematics.

The statements analysed in reverse mathematics are collections of problems (instances)
coming with a class of solutions. Sometimes, it happens that one problem is maximally difficult.
In this case, the strength of the whole statement can be understood by studying this particular instance.

\begin{definition}[Universal instance]
A computable~$\Psf$-instance $X$ is~\emph{universal} if for every computable $\Psf$-instance~$Y$,
every solution to~$X$ computes a solution to~$Y$.
A degree~$\dbf$ is~\emph{$\Psf$-bounding relative to a degree~$\ebf$}
if every $\ebf$-computable instance has a $\dbf$-computable solution. The degree~$\dbf$
is \emph{$\Psf$-bounding} if it is $\Psf$-bounding relative to~$\mathbf{0}$.
\end{definition}

$\wkl$ is known to admit universal instances,
e.g., the tree whose paths are completions of Peano arithmetic,
whereas Mileti~\cite{Mileti2004Partition} proved that $\rt^2_2$ and~$\srt^2_2$ do not admit one.
The author~\cite{Patey2015Degrees} studied extensively which theorems
in reverse mathematics admit a universal instance,
and which do not. It happens that most consequences of $\rt^2_2$
in reverse mathematics do not admit a universal instance.
The most notable exceptions are the rainbow Ramsey theorem 
for pairs~\cite{Csima2009strength,Patey2015Somewhere,WangSome,Wang2014Cohesive},
the finite intersection property~\cite{Cholak2015Any,diamondstone2016finite,dzhafarov2013strength} and~$\dnr$.
It is natural to wonder, given the fact that~$\rwkl$ is a consequence of both~$\srt^2_2$
and of~$\wkl$, whether~$\rwkl$ admits a universal instance.

\begin{question}\label{quest:rwkl-universal-instances}
Does~$\rwkl$ admit a universal instance?
\end{question}

There is a close link between the $\Psf$-bounding degrees and the existence of a universal $\Psf$-instance.
Indeed, the degrees of the solutions to a universal $\Psf$-instance are $\Psf$-bounding.
Using the contraposition, one usually proves that a statement~$\Psf$ admits no universal instance
by showing that every computable $\Psf$-instance has a solution of degree belonging to a class~$\Ccal$,
and that for every degree~$\dbf \in \Ccal$, there is a computable $\Psf$-instance to which~$\dbf$
bounds no solution~\cite{Mileti2004Partition,Patey2015Degrees}.

Interestingly, Question~\ref{quest:rwkl-universal-instances} has some connections with the~$\srt^2_2$ vs $\coh$ question.
The only construction of solutions to instances of~$\rwkl$ which do not produce
sets of PA degree are variants of Mathias forcing which produce solutions
to~$\seqs{\rwkl}$. In both cases, the solutions are~$\rwkl$-bounding,
and in the latter case, they have a jump of PA degree relative to~$\emptyset'$.
By the previous discussion, one should not expect to prove that~$\rwkl$
admits no universal instance in the usual way, unless by finding a new forcing notion.

Some statements do not admit a universal instance because their class of instances is too
restrictive, but they have a natural strengthening which does admit one.
This is for example the case of the rainbow Ramsey theorem for pairs,
which by Miller~\cite{MillerAssorted} admit a universal instance whose solutions are of
DNC degree relative to~$\emptyset'$, but stable variants of the rainbow Ramsey theorem
do not admit one~\cite{Patey2015Degrees}. It is therefore natural to wonder whether there is some 
strengthening of~$\rwkl$ still below~$\rt^2_2$ which admits a universal instance. The Erd\H{o}s-Moser theorem,
defined in the next section, is a good candidate. The following question is a weakening of this interrogation.

\begin{question}
Is there some instance of~$\rt^1_2$ or some computable
instance of~$\rt^2_2$ whose solutions are $\rwkl$-bounding?
\end{question}

Weak K\"onig's lemma is equivalent to several theorems over~$\rca$~\cite{Simpson2009Subsystems}.
However, the Ramsey-type versions of those theorems do not always
give statements equivalent to~$\rwkl$ over~$\rca$. For instance, 
the existence of a separation of two computably inseparable c.e.\ sets
is equivalent to~$\wkl$ over~$\rca$, whereas it is easy to compute
an infinite subset of a separating set.
Among those statements, the Ramsey-type version of the graph coloring problem~\cite{Hirst1990Marriage}
is of particular interest in our study of~$\rwkl$. Indeed,
the $\rwkl$-instances built in~\cite{bienvenu2017logical,flood2016separating,Patey2015Ramsey} are only of two kinds:
trees of positive measure, and trees whose paths code the $k$-colorings
of a computable $k$-colorable graph. The restriction of~$\rwkl$ to the former class of instances
is equivalent to~$\dnr$ over~$\rca$. We shall discuss further the latter one.

\begin{definition}[Ramsey-type graph coloring]{\ }
Let $G = (V,E)$ be a graph.  A set $H \subseteq V$ is {\itshape $k$-homogeneous for $G$} if every finite $V_0 \subseteq V$ induces a subgraph that is $k$-colorable by a coloring that colors every vertex in $V_0 \cap H$ color~$0$.
$\rcolor_k$ is the statement  ``for every infinite, locally $k$-colorable graph $G = (V,E)$, there is an infinite $H \subseteq V$ that is $k$-homogeneous for $G$.''
\end{definition}

The Ramsey-type graph coloring statements have been introduced by Bienvenu et al.~\cite{bienvenu2017logical}.
They proved by an elaborate construction that~$\rcolor_n$ is equivalent to~$\rwkl$ for~$n \geq 3$,
and that~$\dnr$ (or even weak weak K\"onig's lemma) does not imply~$\rcolor_2$ over~$\omega$-models.

\begin{question}\label{quest:rcolor2-rwkl}
Does~$\rcolor_2$ imply~$\rwkl$ or even~$\dnr$ over~$\rca$?
\end{question}

Separating~$\rcolor_2$ from~$\rwkl$ may require constructing an~$\rwkl$-instance of a new kind,
i.e., but not belonging to any of the two classes of instances mentioned before.
The question about the existence of a universal instance also holds for~$\rcolor_2$.

%% file: parts/part3-em.tex
The Erd\H{o}s-Moser theorem is a statement from graph theory.
It provides together with the ascending descending principle~($\ads$) an alternative decomposition of
Ramsey's theorem for pairs. Indeed, every coloring~$f : [\Nb]^2 \to 2$
can be seen as a tournament~$R$ such that~$R(x,y)$ holds if~$x < y$ and~$f(x,y) = 1$, or~$x > y$ and~$f(y, x) = 0$.
Every infinite transitive subtournament induces a linear order whose infinite ascending or descending
sequences are $f$-homogeneous.

\begin{definition}[Erd\H{o}s-Moser theorem]
A tournament $T$ on a domain $D \subseteq \Nb$ is an irreflexive binary relation on~$D$ such that for all $x,y \in D$ with $x \not= y$, exactly one of $T(x,y)$ or $T(y,x)$ holds. A tournament $T$ is \emph{transitive} if the corresponding relation~$T$ is transitive in the usual sense. A tournament $T$ is \emph{stable} if $(\forall x \in D)[(\forall^{\infty} s) T(x,s) \vee (\forall^{\infty} s) T(s, x)]$.
$\emo$ is the statement ``Every infinite tournament $T$ has an infinite transitive subtournament.''
$\semo$ is the restriction of $\emo$ to stable tournaments.
\end{definition}

Bovykin and Weiermann~\cite{Bovykin2005strength} introduced the Erd\H{o}s-Moser theorem in reverse mathematics
and proved that $\emo + \ads$ is equivalent to $\rt^2_2$ over $\rca$.
Lerman et al.~\cite{Lerman2013Separating} proved that~$\emo$ is strictly weaker than~$\rt^2_2$ over~$\omega$-models.
This separation has been followed by various refinements of the weakness of~$\emo$
by Wang~\cite{wang2016definability} and the author~\cite{Patey2013note,patey2017dominating}.
Bienvenu et al.~\cite{bienvenu2017logical} and Flood \& Towsner~\cite{flood2016separating} 
independently proved that~$\semo$ strictly implies~$\rwkl$ over~$\rca$.

Cholak et al.~\cite{Cholak2001strength} proved that every computable $\rt^2_2$-instance has a low${}_2$ solution,
while Mileti~\cite{Mileti2004Partition} and the author~\cite{Patey2015Degrees} showed that various consequences~$\Psf$ of~$\rt^2_2$
do not have~$\Psf$-bounding low${}_2$ degrees, showing therefore that~$\Psf$ does not have
a universal instance. This approach does not apply to~$\emo$ since there is a low${}_2$ $\emo$-bounding degree~\cite{Patey2015Degrees}.
The Erd\H{o}s-Moser theorem is, together with~$\rwkl$, one of the last Ramsey-type theorems 
for which the existence of a universal instance is unknown. A positive answer to the following question
would refine our understanding of $\rwkl$-bounding degrees. Note that, like $\rwkl$, the only known
forcing notion for building solutions to $\emo$-instances produces $\emo$-bounding degrees.

\begin{question}
Does $\emo$ admit a universal instance?
\end{question}

\update{This question was answered negatively by Liu and Patey in \cite{liu2022reverse}.}
\smallskip

Due to the nature of the decomposition of~$\rt^2_2$ into~$\emo$ and~$\ads$,
the Erd\H{o}s-Moser theorem shares many features with~$\rt^2_2$. In particular,
there is a computable $\semo$-instance with no low solution~\cite{Kreuzer2012Primitive}.
The forcing notion used to construct solutions to $\emo$-instances is very similar to the
one used to construct solutions to~$\rt^2_2$-instances. The main difference is that
in the~$\emo$ case, only one object (a transitive subtournament) is constructed,
whereas in the~$\rt^2_2$ case, both a set homogeneous with color~0 and a set homogeneous with color~1
are constructed. As a consequence, the constructions in the $\emo$ case remove the disjunction
appearing in almost every construction of solutions to~$\rt^2_2$,
and therefore simplifies many arguments, while preserving some computational power.
In particular, the author proved that~$\coh \leq_c \srt^2_2$ if and only if~$\coh \leq_c \semo$.

Considering that~$\semo$ behaves like~$\srt^2_2$ with respect to~$\coh$,
one would wonder whether $\emo$, like~$\rt^2_2$, implies~$\coh$ over~$\rca$.
The closest result towards an answer is the proof that~$\emo$
implies~$[\sts^2 \vee \coh]$ over~$\rca$, where~$\sts^2$ is the restriction of~$\ts^2$
to stable functions~\cite{Patey2015Somewhere}. 
The Erd\H{o}s-Moser theorem is known not to imply~$\sts^2$ over $\omega$-models~\cite{Patey2013note,wang2016definability},
but the following question remains open.

\begin{question}
Does~$\emo$ imply~$\coh$ over~$\rca$?
\end{question}

A natural first step in the study of the computational strength of a principle
consists in looking at how it behaves with respect to ``typical'' sets.
Here, by typical, we mean randomness and genericity. By Liu~\cite{Liu2015Cone},
$\emo$ does not imply the existence of a Martin-Löf random. However, it implies the existence
of a function DNC relative to~$\emptyset'$~\cite{Kang2014Combinatorial,Patey2015Somewhere}, 
which corresponds to the computational power of an infinite subset of a 2-random. 
On the genericity hand, $\emo$ implies the existence of a hyperimmune set~\cite{Lerman2013Separating}
and does not imply $\Pi^0_1$-genericity~\cite{Hirschfeldt2009atomic,patey2017dominating}.
The relation between~$\emo$ and 1-genericity is currently unclear.

\begin{definition}[Genericity]
Fix a set of strings~$S \subseteq 2^{<\Nb}$.
A real~$G$ \emph{meets}~$S$ if it has some initial segment in~$S$.
A real~$G$ \emph{avoids}~$S$ if it has an initial segment with no extension in~$S$.
Given an integer~$n \in \Nb$, a real is~\emph{$n$-generic} if it meets or avoids each~$\Sigma^0_n$ set of strings.
$\gens{n}$ is the statement ``For every set~$X$, there is a real $n$-generic relative to~$X$''.
\end{definition}

1-genericity received particular attention from the reverse mathematical community recently as it happened
to have close connections with the finite intersection principle~($\fip$).
Dazhafarov \& Mummert~\cite{dzhafarov2013strength} introduced~$\fip$ in reverse mathematics.
Day, Dzhafarov, and Miller [unpublished] and Hirschfeldt and Greenberg [unpublished] 
independently proved that it is a consequence of the atomic
model theorem (and therefore of~$\srt^2_2$) over~$\rca$.
Diamondstone et al.~\cite{diamondstone2016finite} first established a link between 1-genericity and~$\fip$
by proving that~$\gens{1}$ implies~$\fip$ over~$\rca$. Later, Cholak et al.~\cite{Cholak2015Any} proved the
other direction.

\begin{question}\label{quest:emo-genericity-rca}
Does~$\emo$ imply~$\gens{1}$ over~$\rca$?
\end{question}

Every 1-generic bounds an $\omega$-model of~$\gens{1}$.
This property is due to a Van Lambalgen-like theorem for genericity~\cite{Yu2006Lowness},
and implies in particular that there is no 1-generic of minimal degree.
Cai et al.~\cite{Cai2015DNR} constructed an $\omega$-model of~$\dnr$
such that the degrees of the second-order part belong to a sequence~$\mathbf{0}, \dbf_1, \dbf_2, \dots$
where~$\dbf_1$ is a minimal degree and~$\dbf_{n+1}$ is a strong minimal cover of~$\dbf_n$.
By the previous remark, such a model cannot contain a 1-generic real. Construcing a similar
model of~$\emo$ would answer negatively Question~\ref{quest:emo-genericity-rca}.
The following question is a first step towards an answer. In particular, 
answering positively would provide a proof that~$\gens{1} \not \leq_c \emo$.

\begin{question}
Does every computable~$\emo$-instance (or even~$\rwkl$-instance) admit 
a solution of minimal degree?
\end{question}

%% file: parts/part4-hierarchies.tex
Jockusch~\cite{Jockusch1972Ramseys} proved that the hierarchy
of Ramsey's theorem collapses at level 3 over $\omega$-models, that is,
for every~$n, m \geq 3$, $\rt^n_2$ and~$\rt^m_2$ have the same $\omega$-models. 
Simpson~\cite[Theorem III.7.6]{Simpson2009Subsystems} formalized Jockusch's results
within reverse mathematics and proved that $\rt^n_2$
is equivalent to the arithmetic comprehension axiom ($\aca$) over~$\rca$ for every~$n \geq 3$.
Since~$\rt^1_2$ is provable over~$\rca$ and~$\rt^2_2$ is strictly between~$\rt^3_2$ and~$\rca$,
the status of the whole hierarchy of Ramsey's theorem is known.

However, some consequences of Ramsey's theorem form hierarchies
for which the question of the strictness is currently unanswered.
This is for example the case of the thin set theorems, the free set theorem
and the rainbow Ramsey theorem.
The thin set theorems, already introduced, are natural weakenings
of Ramsey's theorem, where the solutions are allowed to use more than one color.
The free set theorem is a strengthening of the thin set theorem over~$\omega$ colors,
in which every member of a free set is a witness of thinness of the same set.
Indeed, if~$H$ is an infinite~$f$-free set for some function~$f$,
for every~$a \in H$, $H \setminus \{a\}$ is $f$-thin with witness color~$a$.
See Theorem~3.2 in~\cite{Cholak2001Free} for a formal version of this claim.

\begin{definition}[Free set theorem]
Given a coloring $f : [\Nb]^n \to \Nb$, an infinite set~$H$
is \emph{free} for~$f$ if for every~$\sigma \in [H]^n$,
$f(\sigma) \in H \imp f(\sigma) \in \sigma$.
For every~$n \geq 1$, $\fs^n$ is the statement
``Every coloring $f : [\Nb]^n \to \Nb$ has a free set''.
$\fs$ is the statement~$(\forall n)\fs^n$.
\end{definition}

The free set theorem has been introduced by Friedman~\cite{FriedmanFom53free}
together with the thin set theorem.
Cholak et al.~\cite{Cholak2001Free}
proved that $\ts^n$ is a consequence of~$\fs^n$ for every~$n \geq 2$.
Wang~\cite{Wang2014Some} proved that the full free set hierarchy (hence the thin set hierarchy) lies strictly below~$\aca$
over~$\rca$. This result was improved by the author~\cite{patey2015combinatorial} who
showed that $\fs$ does not even imply~$\wkl$ (and in fact weak weak König's lemma) over~$\rca$.

\begin{question}
Does the free set theorem (resp.\ the thin set theorem with $\omega$ colors) form a strict hierarchy?
\end{question}

Jockusch~\cite{Jockusch1972Ramseys} proved for every~$n \geq 2$ that every computable~$\rt^n_2$-instance
has a~$\Pi^0_n$ solution and constructed a computable~$\rt^n_2$-instance with no $\Sigma^0_n$ solution.
Cholak et al.~\cite{Cholak2001Free} proved that $\fs^n$ and~$\ts^n$ satisfy the same bounds.
In particular, there is no $\omega$-model of $\fs^n$ or~$\ts^n$ containing only~$\Delta^0_n$ sets.
By Cholak et al.~\cite{Cholak2001strength}, every computable instance of~$\fs^2$ and~$\ts^2$
admit a low${}_2$ solution. Therefore $\fs^2$ (hence $\ts^2$) are strictly weaker than~$\ts^3$ (hence~$\fs^3$) over~$\rca$.
Using this approach, a way to prove the strictness of the hierarchies would be to answer positively the following question.

\begin{question}\label{quest:fs-ts-lown-solution}
For every~$n \geq 3$, does every computable instance of~$\fs^n$ (resp.\ $\ts^n$)
admit a low${}_n$ solution?
\end{question}

The complexity of controlling the $n$th iterate of the Turing jump grows very quickly with~$n$.
While the proof of the existence of low${}_2$ solutions to computable $\rt^2_2$-instances 
using first jump control is relatively simple,
the proof using the second jump control requires already a much more elaborate framework~\cite{Cholak2001strength}.
Before answering Question~\ref{quest:fs-ts-lown-solution}, one may want to get rid of the technicalities
due to the specific combinatorics of the free set and thin set theorems, and focus
on the control of iterated Turing jumps by constructing simpler objects.

Non-effective instances of Ramsey's theorem for singletons is a good starting point,
since the only combinatorics involved are the pigeonhole principle.
Moreover, $\rt^1_2$ can be seen as a bootstrap principle, above which the other Ramsey-type
statements are built. For instance, cohesiveness is proven by making $\omega$ applications
of the~$\rt^1_2$ principles,
and $\rt^2_2$ is obtained by making one more application of~$\rt^1_2$ over a non-effective instance.
The proofs of the free set and thin set theorems also make an important use of non-effective instances of~$\rt^1_2$~\cite{patey2015combinatorial,Wang2014Some}.

\begin{question}
For every~$n \geq 3$, does every~$\Delta^0_n$ set admit an infinite low${}_n$ subset of either it or its complement?
\end{question}

The solutions to Ramsey-type instances are usually built by forcing.
In order to obtain low${}_n$ solutions, one has in particular 
to $\emptyset^{(n)}$-effectively decide the $\Sigma^0_n$ theory of the generic set.
The forcing relation over a partial order is defined inductively,
and intuitively expresses whether, whatever the choices of conditions extensions we make
in the future, we will still be able to make some progress in satisfying the considered property.

This raises the problem of talking about the ``future'' of a condition~$c$. 
To do that, one needs to be able to describe effectively the conditions extending~$c$.
The problem of the forcing notions used to build solutions to Ramsey-type instances
is that they use variants of Mathias forcing, whose conditions cannot be described effectively.
For instance, let us take the simplest notion of Mathias forcing: pairs~$(F, X)$
where $F$ is a finite set of integers representing the finite approximation of the solution, 
$X$ is a computable infinite ``reservoir'' of integers and~$max(F) < min(X)$. Given a condition~$c = (F, X)$,
the extensions of~$c$ are the pairs~$(H, Y)$ such that~$F \subseteq H$, $Y \subseteq X$
and~$H \setminus F \subseteq X$. Deciding whether a Turing index is the code of 
an infinite computable subset of a fixed computable set requires
a lot of computational power. Cholak et al.~\cite{Cholak2014Generics} studied computable Mathias forcing
and proved that the forcing relation for deciding $\Sigma^0_n$ properties is not~$\Sigma^0_n$ in general.
This is why we need to be more careful in the design of the forcing notions.

In some cases, the reservoir has a particular shape. 
Through their second jump control, Cholak et al.~\cite{Cholak2001strength} first used this idea by noticing
that the only operations over the reservoir were partitioning and finite truncation.
This idea has been reused by Wang~\cite{Wang2014Cohesive} to prove that every computable
$\seqs{\mathsf{D}^2_2}$-instance has a solution of low${}_3$ degree.
Recently, the author~\cite{patey2017controlling} designed new forcing notions
for various Ramsey-type statements, e.g., $\coh$, $\emo$ and $\mathsf{D}^2_2$, in which the forcing relation for
deciding~$\Sigma^0_n$ properties is~$\Sigma^0_n$.

%% file: parts/part9-appendix.tex
In this last section, we provide the proofs supporting
some claims which have been made throughout the discussion.
We first clarify the links between sequential variants
of the thin set theorem and diagonally non-computable functions.

\begin{lemma}\label{lem:seqts1-computed-by-dnc}
For every instance $\vec{f}$ of~$\seqs{\ts^1}$
and every set~$C$ whose jump is of DNC degree relative to the jump of~$\vec{f}$,
$C \oplus \vec{f}$ computes a solution to~$\vec{f}$.
\end{lemma}
\begin{proof}
Fix a $\seqs{\ts^1}$-instance $f_0, f_1, \dots : \Nb \to \Nb$.
Let~$A_0, A_1, \dots$ be a uniformly $\vec{f}$-computable sequence of sets
containing~$\Nb$, the sets~$\{ x : f_s(x) \neq i \}$ for each~$s,i \in \Nb$,
and that is closed under the intersection and complementation operations.
Let~$g$ be the partial~$\vec{f}'$-computable function
defined for each~$s, i \in \Nb$ by $g(i,s) = \lim_{n \in A_i} f_s(n)$
if it exists. The jump of~$C$ computes a function~$h(\cdot, \cdot)$
such that~$h(s,i) \neq g(s,i)$ for each~$s, i \in \Nb$.

Let~$e_0, e_1, \dots$ be a $C'$-computable sequence defined inductively by~$A_{e_0} = \Nb$
and~$A_{e_{s+1}} = A_{e_s} \cap \{x : f_s(x) \neq g(e_s,s) \}$.
By definition of~$g$, if~$A_{e_s}$ is infinite, then so is~$A_{e_{s+1}}$.
By Shoenfield's limit lemma~\cite{Shoenfield1959degrees}, the~$e$'s have
a uniformly $C$-computable approximation~$e_{s,0}, e_{s,1}, \dots$ such that~$\lim_t e_{s,t} = e_s$.

We now use the argument of Jockusch \& Stephan in~\cite[Theorem 2.1]{Jockusch1993cohesive}
to build an infinite set~$X$ such that~$X \subseteq^{*} A_{e_s}$ for all~$s \in \Nb$.
Let
\[
x_0 = 0 \hspace{10pt} \mbox{ and } \hspace{10pt} x_{s+1} = min\{x > x_s : (\forall m \leq s)[x \in A_{e_{m, x}}]\}
\]
Assuming that~$x_s$ is defined, we prove that~$x_{s+1}$ is found.
Since~$A_{e_{s+1}} \subseteq A_{e_m}$ for all~$m \leq s$, any sufficiently large element
satisfies~$x > x_s : (\forall m \leq s)[x \in A_{e_{m, x}}]$. Therefore~$x_{s+1}$ is defined.
The set~$X = \{x_0, x_1, \dots\}$ is a solution to~$\vec{f}$ by definition of the sequence $A_{e_0}, A_{e_1}, \dots$
\end{proof}

\begin{lemma}\label{lem:seqts1-jump-dnczp}
Fix a set~$X$ and let~$\vec{f}$ be the instance of~$\seqs{\ts^1}$ consisting
of all $X$-primitive recursive functions. For every solution~$C$ to~$\vec{f}$,
$(X \oplus C)'$ is of DNC degree relative to~$X'$.
\end{lemma}
\begin{proof}
Fix a set~$X$, and 
consider the uniformly $X$-primitive recursive sequence of colorings $f_0, f_1, \dots : \Nb \to \Nb$ defined
for each~$e \in \Nb$ by $f_e(s) = \Phi^{X'_s}_{e,s}(e)$,
where~$X'_s$ is the approximation of the jump of~$X$ at stage~$s$.
Let~$C$ be a solution to the $\seqs{\ts^1}$-instance~$\vec{f}$.
By definition, for every~$e \in \Nb$, there is some~$i$ such that~
$(\forall^{\infty} s)f_e(s) \neq i$. Let~$g : \Nb \to \Nb$
be the function which on input~$e$ makes a~$(X \oplus C)'$-effective search for such~$i$ and outputs it.
The function~$g$ is DNC relative to~$\emptyset'$.
\end{proof}

\begin{theorem}
For every standard~$k$, $\seqs{\ts^1_k} =_c \coh$.
\end{theorem}
\begin{proof}
$\coh \leq_c \seqs{\ts^1_k}$: Let~$R_0, R_1, \dots$ be a $\coh$-instance. 
Let~$f_0, f_1, \dots : \Nb \to k$ be the uniformly $\vec{R}$-computable sequence
of functions defined for each~$e \in \Nb$ by~$f_e(s) = \Phi^{\emptyset'_s}_{e,s}(e) \mod k$.
By slightly modifying the proof of Lemma~\ref{lem:seqts1-jump-dnczp}, for every solution~$C$
to the $\seqs{\ts^1_k}$-instance~$\vec{f}$, $(\vec{R} \oplus C)'$ computes
a $k$-valued function DNC relative to~$\vec{R}'$. 
By Friedberg~\cite{jockusch1989degrees} in a relativized form, $(\vec{R} \oplus C)'$ is of PA degree relative to~$\emptyset'$.
Jockusch \& Stephan~\cite{Jockusch1993cohesive} proved that every set whose jump is of PA degree relative to~$\emptyset'$
computes a solution to~$\vec{R}$. Using this fact, $\vec{R} \oplus C$ computes an infinite~$\vec{R}$-cohesive set.

$\seqs{\ts^1_k} \leq_c \coh$: 
Let~$f_0, f_1, \dots$ be a $\seqs{\ts^1_k}$-instance.
Let~$R_0, R_1, \dots$ be the $\vec{f}$-computable sequence of sets
defined for each~$e,x \in \Nb$ by~$x \in R_e(x)$ if and only if~$f_e(x) = 0$.
Every infinite~$\vec{R}$-cohesive set is a solution to~$\vec{f}$.
\end{proof}

\begin{theorem}
There is an $\omega$-model of~$\seqs{\ts^1}$ that is not a model of~$\coh$.
\end{theorem}
\begin{proof}
By Ku\v{c}era~\cite{Kucera1985Measure} or Jockusch \& Soare~\cite{Jockusch197201} in a relativized form, the measure of oracles
whose jump is of PA degree relative to~$\emptyset'$ is null.
Let~$Z$ be a 2-random (a Martin-L\"of random relative to~$\emptyset'$)
whose jump is not of PA degree relative to~$\emptyset'$.
By Van Lambalgen~\cite{VanLambalgen1990axiomatization}, $Z$ bounds the second-order part of an
$\omega$-model of the statement~``For every set $X$, there is a Martin-L\"of random
relative to~$X'$.''. In particular, by Kjos-Hanssen~\cite{kjoshanssen2009infinite} 
and Greenberg \& Miller~\cite{Greenberg2009Lowness}, $\Mcal \models \dnrs{2}$.
By Lemma~\ref{lem:seqts1-computed-by-dnc}, 
$\Mcal \models \seqs{\ts^1}$, and by Jockusch \& Stephan~\cite{Jockusch1993cohesive},
$\Mcal \not \models \coh$.
\end{proof}

\begin{theorem}\label{thm:ts1-combinatorial-dnc}
For every set~$X$, there is an~$X'$-computable $\ts^1$-instance $f : \Nb \to \Nb$
such that every infinite $f$-thin set computes a function DNC relative to~$X$.
\end{theorem}
\begin{proof}
Fix a set~$X$ and let~$g : \Nb \to \Nb$ be the $X'$-computable 
function such that~$g(e) = \Phi^X_e(e)$ if~$\Phi^X_e(e) \downarrow$,
and~$g(e) = 0$ otherwise. Fix an enumeration of all sets $(D_{e,i} : e, i \in \Nb)$ such that
$D_{e,i}$ is of size~$2\tuple{e,i}+1$, where~$\tuple{e,i}$ is the standard pairing function.
We define our $\ts^1$-instance~$f : \Nb \to \Nb$ by an $X'$-computable 
sequence of finite approximations~$f_0 \subseteq f_1 \subseteq \dots$,
such that at each stage~$s$, $f_s$ is defined on a domain of size at most~$2s$.
Furthermore, to ensure that $f = \bigcup_s f_s$ is $X'$-computable, we require that $s \in dom(f_{s+1})$.

At stage~$0$, $f_0$ is the empty function.
At stage~$s+1$, if~$s = \tuple{e,i}$,
take an element~$x \in D_{g(e),i} \setminus dom(f_s)$, and set~$f_{s+1} = f_s \sqcup \{x \mapsto i\}$
if~$s \in dom(f_s) \cup \{x\}$ and~$f_{s+1} = f_s \sqcup \{x, s \mapsto i \}$ otherwise.
Such an~$x$ must exist as~$\card{D_{g(e),i}} = 2\tuple{e,i}+1 = 2s+1 > dom(f_s)$.
Then go to the next stage.

Let~$H$ be an infinite $f$-thin set, with witness color~$i$.
For each~$e \in \Nb$, let~$h(e)$ be such that~$D_{h(e),i} \subseteq H$.
We claim that~$h(e) \neq \Phi^X_e(e)$. Suppose that it is not the case.
In particular, $h(e) = g(e)$. Let~$s = \tuple{g(e),i}$.
At stage~$s+1$, $f_{s+1}(x) = i$ for some~$x \in D_{g(e),i} = D_{h(e),i} \subseteq H$,
contradicting the fact that~$H$ is $f$-thin.
Therefore~$h$ is a function DNC relative to~$X$.
\end{proof}

\begin{corollary}
For every instance~$\vec{f}$ of~$\seqs{\ts^1}$,
there is an instance~$g$ of~$\ts^1$ such that  
for every solution~$H$ to~$g$, $H \oplus \vec{f}$ computes a solution to~$\vec{f}$.
\end{corollary}
\begin{proof}
Fix a $\seqs{\ts^1}$-instance~$\vec{f}$.
By Theorem~\ref{thm:ts1-combinatorial-dnc}, 
there is a $\ts^1$-instance~$g$ such that every infinite $g$-thin set~$H$ computes
a function DNC relative to~$\vec{f}$.
By Lemma~\ref{lem:seqts1-computed-by-dnc}, $H \oplus \vec{f}$ computes
a solution to~$\vec{f}$.
\end{proof}

The increasing polarized Ramsey's theorem has been introduced by Dzhafarov
and Hirst~\cite{Dzhafarov2009polarized} to find new principles between stable Ramsey's theorem for pairs
and Ramsey's theorem for pairs. We prove that the relativized
Ramsey-type weak K\"onig's lemma and the increasing polarized Ramsey's theorem are equivalent over $\mathsf{RCA}_0$.

\begin{definition}[Relativized Ramsey-type weak K\"onig's lemma]
Given an infinite set of strings $S \subseteq 2^{<\mathbb{N}}$,
let~$T_S$ denote the downward closure of $S$, that is,
$T_S = \{ \tau \in 2^{<\mathbb{N}} : (\exists \sigma \in S)[\tau \preceq \sigma] \}$.
$2\mbox{-}\mathsf{RWKL}$ is the statement~``For every set of strings~$S$,
there is an infinite set which is homogeneous for~$T_S$''.
\end{definition}

Note that the statement $\rwkls{2}$ slightly differs from a relativized
variant of~$\rwkl$ where the tree would have a $\Delta^0_2$ presentation.
However, those two formulations are equivalent over $\rca+\bst$
(see Flood~\cite{Flood2012Reverse}),
and therefore over~$\rca$ since~$\bst$ is a consequence of both statements.

\begin{definition}[Increasing polarized Ramsey's theorem]
A set~\emph{increasing p-homogeneous} for $f : [\mathbb{N}]^n \to k$
is a sequence~$\langle H_1, \dots, H_n \rangle$ of infinite sets such that
for some color~$c < k$, $f(x_1, \dots, x_n) = c$ for every increasing
tuple $\langle x_1, \dots, x_n \rangle \in H_1 \times \dots \times H_n$.
$\mathsf{IPT}^n_k$ is the statement~``Every coloring~$f : [\mathbb{N}]^n \to k$
has an infinite increasing p-homogeneous set''.
\end{definition}

\begin{theorem}
$\mathsf{RCA}_0 \vdash \mathsf{IPT}^2_2 \leftrightarrow 2\mbox{-}\mathsf{RWKL}$
\end{theorem}
\begin{proof}
$\mathsf{IPT}^2_2 \rightarrow 2\mbox{-}\mathsf{RWKL}$:
Let~$S = \{\sigma_0, \sigma_1, \dots \}$ be an infinite set of strings
such that~$|\sigma_i| = i$ for each~$i$.
Define the coloring~$f : [\mathbb{N}]^2 \to 2$ for each~$x < y$ by $f(x, y) = \sigma_y(x)$.
By~$\mathsf{IPT}^2_2$, let~$\langle H_1, H_2 \rangle$ be an infinite set increasing p-homogeneous for~$f$ with some color~$c$.
We claim that~$H_1$ is homogeneous for~$T_S$ with color~$c$. We will
prove that the set $I = \{\sigma \in T_S : \text{$H_1$ is homogeneous for $\sigma$}\}$ is infinite.
For each~$y \in \mathbb{N}$, let~$\tau_y$ be the string of length~$y$ defined by
$\tau_y(x) = f(x,y)$ for each~$x < y$. By definition of~$f$, $\tau_y \in S$ for each~$y \in \mathbb{N}$.
By definition of~$\langle H_1, H_2 \rangle$, $\tau_y(x) = c$ for each~$x \in H_1$ and~$y \in H_2$. 
Therefore, $H_1$ is homogeneous for~$\tau_y$ with color~$c$ for each~$y \in H_2$.
As $\{\tau_y : y \in H_2 \} \subseteq I$, the set~$I$ is infinite and therefore~$H_1$
is homogeneous for~$T_S$ with color~$c$.

$2\mbox{-}\mathsf{RWKL} \rightarrow \mathsf{IPT}^2_2$:
Let~$f : [\mathbb{N}]^2 \to 2$ be a coloring.
For each~$y$, let~$\sigma_y$ be the string of length~$y$
such that~$\sigma_y(x) = f(x, y)$ for each~$x < y$,
and let~$S = \{ \sigma_i : i \in \mathbb{N} \}$.
By~$2\mbox{-}\mathsf{RWKL}$, let~$H$ be an infinite set homogeneous for~$T_S$ with some color~$c$.
Define~$\langle H_1, H_2 \rangle$ by stages as follows.
At stage~0, $H_{1,0} = H_{2,0} = \emptyset$.
Suppose that at stage~$s$, $|H_{1,s}| = |H_{2,s}| = s$, $H_{1,s} \subseteq H$
and $\langle H_{1,s}, H_{2,s} \rangle$ is a finite set increasing p-homogeneous for~$f$ with color~$c$.
Take some~$x \in H$ such that~$x > max(H_{1,s}, H_{2,s})$ and set~$H_{1,s+1} = H_{1,s} \cup \{x\}$.
By definition of~$H$, there exists a string $\tau \prec \sigma_y$ for some~$y > x$,
such that $|\tau| > x$ and~$H$ is homogeneous for $\tau$ with color~$c$. Set~$H_{2,s+1} = H_{2,s} \cup \{y\}$.
We now check that the finite set $\langle H_{1,s+1}, H_{2,s+1} \rangle$ is an increasing p-homogeneous for~$f$ with color~$c$.
By induction hypothesis, we need only to check that $f(z, y) = c$ for every~$z \in H_{1,s+1}$. 
By definition of homogeneity
and as~$H_{1,s+1} \subset H$, $\sigma_y(z) = c$ for every~$y \in H_{1,s+1}$. By definition of~$\sigma_y$,
$f(z, y) = c$ for every~$z \in H_{1,s+1}$. 
This finishes the proof.
\end{proof}